\documentclass[11pt]{article}
\usepackage[utf8]{inputenc}
\usepackage[dvipsnames]{xcolor}
\usepackage{graphicx}
\usepackage{amsmath}
\usepackage{amsfonts}
\usepackage{amssymb}
\usepackage{amsthm}
\usepackage{thmtools}
\usepackage[margin=25mm]{geometry}

\usepackage{caption}
\usepackage{enumitem,calc}
\usepackage[longnamesfirst,numbers,sort&compress]{natbib}

\usepackage[unicode=true]{hyperref}
\hypersetup{
colorlinks,
linkcolor={black},
citecolor={black},
urlcolor={blue!60!black}
}

\usepackage[noabbrev,capitalise,nameinlink]{cleveref}
\crefname{lemma}{Lemma}{Lemmas}
\crefname{theorem}{Theorem}{Theorems}
\crefname{corollary}{Corollary}{Corollaries}
\crefname{proposition}{Proposition}{Propositions}
\crefname{figure}{Figure}{Figures}

\theoremstyle{plain}
\newtheorem{theorem}{Theorem}

\newtheorem{lemma}[theorem]{Lemma}
\newtheorem{corollary}[theorem]{Corollary}

\newcounter{proofcount}
\AtBeginEnvironment{proof}{\stepcounter{proofcount}}
\newtheorem{claim}{Claim}
\makeatletter                  
\@addtoreset{claim}{proofcount}
\makeatother

\theoremstyle{definition}

\DeclareMathOperator{\dist}{dist}

\DeclareMathOperator{\ANdim}{ANdim}

\newcommand{\scr}[1]{\mathcal{#1}}
\newcommand{\ds}[1]{\mathbb{#1}}

\newcommand{\defn}[1]{\textcolor{Maroon}{\emph{#1}}}

\renewcommand{\geq}{\geqslant}
\renewcommand{\leq}{\leqslant}

\begin{document}

\title{An improved quasi-isometry between graphs of bounded cliquewidth and graphs of bounded treewidth}
\author{Marc Distel\,\footnotemark[2]}
\date{\today}

\footnotetext[2]{School of Mathematics, Monash University, Melbourne, Australia (\texttt{marc.distel@monash.edu}). Research supported by Australian Government Research Training Program Scholarship.}
\maketitle

\begin{abstract}
    Cliquewidth is a dense analogue of treewidth. It can be deduced from recent results by Hickingbotham [arXiv:2501.10840] and Nguyen, Scott, and Seymour [arXiv:2501.09839] that graphs of bounded cliquewidth are quasi-isometric to graphs of bounded treewidth. We improve on this by showing that graphs of cliquewidth $k$ admit a partition with `local, but dense' parts whose quotient has treewidth $k-1$. Specifically, each part is contained within the closed neighbourhood of some vertex. We use this to construct a $3$-quasi-isometry between graphs of cliquewidth $k$ and graphs of treewidth $k-1$. This is an improvement in both the quasi-isometry parameter and the treewidth. We also show that the bound on the treewidth is tight up to an additive constant. 
\end{abstract}


\setlength{\parindent}{0pt}
\setlength{\parskip}{10pt}

\section{Introduction}

    The treewidth of a graph is a measure of how far it is from a tree, and is especially  important in structural and algorithmic graph theory. It is only bounded on sparse graphs; that is, graphs with large average degree have large treewidth.

    Cliquewidth is an analogue of treewidth for dense graphs. Unlike treewidth, cliquewidth can be bounded on dense graphs. In fact, cliquewidth is only meaningfully distinct from treewidth for dense graphs, since \citet{Corneil2005} showed that if a graph excludes $K_{t,t}$ as a subgraph and has cliquewidth at most $k$, then it has treewidth at most $3k(t-1)-1$. For the reverse direction, \citet{Corneil2005} showed that any graph (sparse or dense) of treewidth at most $k$ has cliquewidth at most $3\cdot 2^{k-1}$ (see \citet{Courcelle2000B} for an earlier bound).

    This analogy can be furthered using the concept of `quasi-isometry'. Graphs are said to be `quasi-isometric' if they `look similar at a large scale'. Quasi-isometries have a parameter $c$ that controls just how strong the similarity is, with smaller values indicating a greater degree of similarity. This is denoted as a $c$-quasi-isometry. See \cref{secQI} for a formal definition.

    It can be concluded from the recent results of  \citet{Hickingbotham2025} and \citet{Nguyen2025} that 
    graphs of bounded cliquewidth are quasi-isometric to graphs of bounded treewidth. In particular, every graph of cliquewidth $k$ is $(4k+4)$-quasi-isometric to a graph of treewidth at most $6k$. This is obtained by using a different parameter called `sim-width', which is at most the clique-width (see \citep[Proposition~6.3]{Oum2006} and \citep[Lemma~2.2]{Kang2017}). Combined with \citep[Lemma~13]{Hickingbotham2025} and \citep[(1.2)]{Nguyen2025}, it can be deduced that each graph of cliquewidth $k$ is $f(k)$-quasi-isometric to a graph of treewidth at most $6k$. A slightly deeper dive into the proof of \citep[(1.2)]{Nguyen2025} (specifically, using \citep[(2.3)]{Nguyen2025} and the observation in the paragraph preceding it with $r=2$) reveals that $f(k)=4k+4$ suffices.

    We improve on this result in the following ways. (1) we provide a direct and algorithmic construction using just the definitions of cliquewidth and treewidth, (2) the quasi-isometry parameter is smaller and independent of $k$, and (3) the bound on the treewidth is smaller.

    \begin{theorem}
        \label{QIcwtw}
        For any integer $k\geq 1$, each graph of cliquewidth at most $k$ is $3$-quasi-isometric to a graph of treewidth at most $k-1$.
    \end{theorem}

    We also show that the bound on the treewidth in \cref{QIcwtw} is within an additive constant of being tight.

    \begin{theorem}
        \label{lowerBound}
        For every real number $c\geq 1$ and integer $k\geq 6$, there exists a graph of cliquewidth at most $k$ that is not $c$-quasi-isometric to any graph of treewidth less than $k-3$.
    \end{theorem}

    We actually prove the following strengthening of \cref{QIcwtw}.

    \begin{theorem}
        \label{main}
        For any integer $k\geq 1$, each graph $G$ of cliquewidth at most $k$ admits a dominated partition $\scr{P}$ such that $G/\scr{P}$ has treewidth at most $k-1$.
    \end{theorem}

    Here, a dominated partition is one where each part is within the closed neighbourhood of some vertex. Hence, each part is contained within a subgraph of small radius. We consider such partitions for dense graphs to be analogous to bounded width partitions for sparse graphs.

    \cref{main} directly implies \cref{QIcwtw}; see \cref{secQI} for the (brief) proof.

    We end this paper with an application of \cref{QIcwtw} to a parameter called `Assouad--Nagata dimension', which describes the `large-scale shape' of a graph. See \cref{secANdim} for a formal definition. This application does not use the improvements given by \cref{QIcwtw}; in fact it can be derived just from the work of \citet{Hickingbotham2025} and \citet{Nguyen2025}. However, it was not observed in either of these works (with \citet{Hickingbotham2025} instead observing an application to the related, but weaker, parameter `asymptotic dimension'), and is interesting, so we think it appropriate to include here. The improved quasi-isometry also leads to an improved bound within the definition of Assouad--Nagata dimension (see `control functions'), which may be useful.

    The following was shown independently by \citet{Distel2023proper} and \citet{Liu2025}.

    \begin{lemma}[\citet{Distel2023proper,Liu2025}]
        \label{ANdimTw}
        For any integer $k\geq 0$, the class of all graphs of treewidth at most $k$ has Assouad--Nagata dimension 1.
    \end{lemma}

    See \citet{Benjamini2012} for an implicit proof of \cref{ANdimTw} but with a bounded degree restriction. See also \citet{Bonamy2024} for a proof of \cref{ANdimTw}, but with `Assouad-Nagata dimension' replaced by `asymptotic dimension'. We remark that the Assouad--Nagata dimension is at most the asymptotic dimension.

    We use \cref{ANdimTw} and \cref{QIcwtw} to prove the following.

    \begin{theorem}
        \label{ANdimCw}
        For any integer $k\geq 3$, the class of all graphs of cliquewidth at most $k$ has Assouad--Nagata dimension 1.
    \end{theorem}

    \section{Initial Definitions}

    A \defn{graph} $G$ consists of a vertex set $V(G)$ and an edge set $E(G)\subseteq \binom{V(G)}{2}$. We exclude parallel edges and loops. A graph $H$ is \defn{subgraph} of $G$, denoted \defn{$H\subseteq G$}, if $V(H)\subseteq V(G)$ and $E(H)\subseteq E(G)$. A subgraph $H$ is \defn{proper} if $H\neq G$. The subgraph \defn{induced} by $S\subseteq V(G)$, denoted \defn{$G[S]$}, is the graph with vertex set $V(G)\cap S$ and every edge $uv\in E(G)$ with $u,v\in V(G)\cap S$.

    The \defn{closed neighbourhood} of a vertex $v$ in a graph $G$, $N_G[v]$, consists the set of vertices $u\in V(G)$ such that either $u=v$ or $u$ is adjacent to $v$ in $G$. A set $S\subseteq V(G)$ is \defn{dominated} in $G$ if $S\subseteq N_G[v]$ for some $v\in V(G)$.

    A \defn{partition} $\scr{P}$ of a graph $G$ is a collection of nonempty pairwise disjoint subsets of $V(G)$ such that $\bigcup_{P\in \scr{P}} P=V(G)$. The subsets $P\in \scr{P}$ are called the \defn{parts} of $\scr{P}$. Note that we do not require that parts induce connected subgraphs of $G$. The \defn{quotient} of $G$ with respect to $\scr{P}$, denoted \defn{$G/\scr{P}$}, is the graph with vertex set $\scr{P}$ and edges between parts whenever they are adjacent in $G$. We say that $\scr{P}$ is dominated in $G$ if every part is dominated in $G$. 

    Given a graph $G$ and a set $C$, a \defn{colouring} of $G$ with \defn{colour-set} $C$ is a map $c:V(G)\mapsto C$. By convention, we usually take $C:=\{1,\dots,|C|\}$.  Say each vertex $v\in V(G)$ is \defn{coloured} $c(v)$. Say that a colour $i\in C$ is \defn{used} by $c$ if $c^{-1}(i)\neq\emptyset$. A set $S\subseteq V(G)$ is \defn{monochromatic} under $c$, or \defn{$c$-monochromatic}, if all vertices of $S$ are coloured the same colour under $c$. Say that $S$ is \defn{rainbow} under $c$ if, for each $i\in C$ that is used by $c$, there exists $v\in S$ coloured $i$ under $c$. Say that a partition $\scr{P}$ of $G$ is \defn{monochromatic} under $c$ if each part in $\scr{P}$ is monochromatic under $c$. Note that if $\scr{P}$ is monochromatic, then there is an induced colouring on the quotient $G/\scr{P}$, which we denote \defn{$c_{\scr{P}}$}, and it makes sense to talk about the colour of a part under $c$.

    Let $G,H$ be graphs with $V(G)\cap V(H)=\emptyset$. The \defn{disjoint union} of $G$ and $H$ is the graph with vertex set $V(G)\cup V(H)$ and edge set $E(G)\cup E(H)$. Given colourings $c_G,c_H$ of $G$ and $H$ respectively on the same set of colours, define a colouring \defn{$c_G\cup c_H$} of $G\cup H$ as follows. For each $v\in V(G\cup H)$, let $X$ be the unique element of $\{G,H\}$ such that $v\in V(X)$, and let $c_G\cup c_H(v):=c_X(v)$.

    A \defn{tree decomposition} of a graph $G$ is a tree $T$ combined with a collection $(B_t:t\in V(T))$ such that:

    \begin{enumerate}[label=(TD\arabic*)]
        \item \label{TD1}for each $v\in V(G)$, the vertices $t\in V(T)$ such that $v\in B_t$ induce a nonempty subtree of $T$, and
        \item \label{TD2}for each $uv\in E(G)$, there exists $t\in V(T)$ such that $u,v\in B_t$
    \end{enumerate}

    The sets $B_t$, $t\in V(T)$, are called \defn{bags}. The \defn{width} of a tree-decomposition is the maximum size of a bag minus 1. The \defn{treewidth} of $G$ is the minimum width of a tree-decomposition of $G$.

    We now move towards defining `cliquewidth'. We preface by saying that there are some slight technical differences between our definition and that of other papers. We do this to simplify technical aspects of the proof.

    Let $G$ be a graph, let $k\geq 1$ be an integer, and let $c$ be a colouring of $G$ with colours $\{1,\dots,k\}$. We say that $(G,c)$ is a \defn{cliquewidth-$k$ pair} if either $|V(G)|\leq 1$, or if at least one of the following is true.

    \begin{enumerate}[label=(OP\arabic*)]
        \item \label{OP1}there exist cliquewidth-$k$ pairs $(G_1,c_1),(G_2,c_2)$ such that $G_1$ and $G_2$ are nonempty, $G$ is the disjoint union of $G_1$ and $G_2$, and $c=c_1\cup c_2$, or,
        \item \label{OP2}there exist a colouring $c'$ of $G$ and distinct $i,j\in \{1,\dots,k\}$ such that $(G,c')$ is a cliquewidth-$k$ pair, $i,j$ are used by $c'$, and $c$ is obtained from $c'$ by changing the colour of every vertex coloured $i$ under $c'$ to $j$, or,
        \item \label{OP3}there exist a proper subgraph $G'$ of $G$ and distinct $i,j\in \{1,\dots,k\}$ such that $(G',c)$ is a cliquewidth-$k$ pair, and $G$ is obtained from $G'$ by adding an edge between each pair of (non-adjacent) vertices $u,v$ with $u$ coloured $i$ and $v$ coloured $j$.
    \end{enumerate}

    We can consider the prior list as a list of `operations' to `build' cliquewidth-$k$ pairs.
    
    It is easy to check that cliquewidth-$k$ pairs are preserved under relabellings of vertices and permutations of the colours. The latter point means that the requirement in \ref{OP2} that $j$ is used by $c'$ could be removed without changing what pairs are cliquewidth-$k$ pairs. We include this requirement because it (along with the  `nonempty' condition in \ref{OP1} and `proper' condition in \ref{OP3}) prevents an infinite loop of operations. Because of this, it makes sense to perform induction on the number of operations used to form a cliquewidth-$k$ pair.

    The \defn{cliquewidth} of a graph $G$ is the minimum $k$ such that $G$ admits a colouring $c$ on the colours $\{1,\dots,k\}$ for which $(G,c)$ is a cliquewidth-$k$ pair.

    \section{Proof of \cref{main}}
    
    \cref{main} is an immediate corollary of the following lemma.

    \begin{lemma}
        For any integer $k\geq 1$ and any cliquewidth-$k$ pair $(G,c)$, $G$ admits a $c$-monochromatic dominated partition $\scr{P}$ and a tree-decomposition $(B_t:t\in V(T))$ of $G/\scr{P}$ of width at most $k-1$ such that:
        \begin{enumerate}[label=(\arabic*)]
            \item \label{strong1}there exists $q\in V(T)$ such that $B_q$ is rainbow under $c_{\scr{P}}$, and,
            \item \label{strong2}for each colour $i\in \{1,\dots,k\}$ that is used by $c$, the set of vertices $t\in V(T)$ such that $B_t$ contains a part coloured $i$ under $c$ induces a nonempty subtree of $T$.
        \end{enumerate}
    \end{lemma}

    \begin{proof}
        We perform induction on the number of operations needed to build $(G,c)$. If none of \ref{OP1}-\ref{OP3} are required, observe that $|V(G)|\leq 1$. In this case, we take $\scr{P}:=\{V(G)\}$, which is trivially dominated in $G$ and $c$-monochromatic. Take $T$ be a tree on a single vertex $t$, and $B_t:=\{V(G)\}$. So $(B_t:t\in V(T))$ is trivially a tree-decomposition of $G/\scr{P}$. Further, $B_t$ is rainbow under $c_{\scr{P}}$, and \ref{strong2} is trivially satisfied.
        
        So we may proceed to the inductive step, where at least one operation was used, and in particular, there is a `last' operation used. Let $(G,c)$ be a cliquewidth-$k$ pair, and assume the lemma is true for cliquewidth-$k$ pairs that can be built with fewer operations. We break into cases depending on what the last operation used to obtain $(G,c)$ is.

        Case \ref{OP1}:

        So there exist cliquewidth-$k$ pairs $(G_1,c_1)$, $(G_2,c_2)$ constructable in fewer operations such that $G_1$ and $G_2$ are nonempty, $G$ is the disjoint union of $G_1$ and $G_2$, and $c=c_1\cup c_2$. As $(G_1,c_1)$ and $(G_2,c_2)$ can both be built with fewer operations, induction applies. For $i\in \{1,2\}$, let $\scr{P}_i$ and $(B^i_t:t\in V(T_i))$ be the partition of $G_i$ and tree-decomposition of $G_i/\scr{P}_i$ respectively obtained via induction. So $\scr{P}_i$ is dominated in $G_i$, is $c_i$-monochromatic, and $(B^i_t:t\in V(T_i))$ has width at most $k-1$ and obeys \ref{strong1} and \ref{strong2}.

        Let $\scr{P}:=\scr{P}_1\cup \scr{P}_2$. So $\scr{P}$ is a partition of $G$, and further it is dominated in $G$ and $c$-monochromatic since, for each $i\in \{1,2\}$, $\scr{P}_i$ is dominated in $G_i$ and is $c_i$-monochromatic. So it remains to find a tree-decomposition of $G/\scr{P}$ that satisfies \ref{strong1} and \ref{strong2}.

        For each $i\in \{1,2\}$, let $q_i\in V(T_i)$ be the vertex that satisfies \ref{strong1}. Let $T$ be the graph obtained from the disjoint union of $T_1$ and $T_2$ by adding a new vertex adjacent to $q_1$ and $q_2$. So $T$ is a tree.
        
        Observe that for each colour $i\in \{1,\dots,k\}$ that is used by $c$, $i$ is used by least one of $c_1$ or $c_2$. Thus, and by definition of $q_1,q_2$, $B^1_{q_1}\cup B^2_{q_2}$ is rainbow under $c_{\scr{P}}$
        Let $B_q$ be a minimal subset of $B^1_{q_1}\cup B^2_{q_2}$ that is rainbow. Note that no two parts in $B_q$ will share a colour, and thus $|B_q|\leq k$. Now, for each $t\in V(T)\setminus \{q\}$, let $B_t:=B^i_t$ for the $i\in \{1,2\}$ such that $t\in V(T_i)$.

        We now show that $(B_t:t\in V(T))$ is a tree-decomposition of $G/\scr{P}$. Note that, subject to this being true, \ref{strong1} follows immediately by definition of $B_q$, and that the width is at most $k-1$ since $|B_q|\leq k$ and since, for each $i\in \{1,2\}$, $(B^i_t:t\in V(T_i))$ has width at most $k-1$. 
        
        \ref{TD2} follows immediately by induction and since $E(G)=E(G_1)\cup E(G_2)$. For \ref{TD1}, observe that for any $i\in \{1,2\}$ and any $P\in \scr{P}_i$, the set $S_P$ of vertices $t\in V(T)$ such that $P\in B_t$ is precisely the set of vertices of $t\in T_i$ such that $P\in B^i_t$, plus possibly $q$. Additionally, observe that $q\in S_P$ only if $q_i\in S_P$. Thus and by induction, $T[S_P]$ is connected and nonempty, as desired. Since $\scr{P}=\scr{P}_1\cup \scr{P}_2$, \ref{TD2} follows. Hence, $(B_t:t\in V(T))$ is a tree-decomposition of $G/\scr{P}$. It remains only to show \ref{strong2}.
        
        For each $i\in \{1,\dots,k\}$, and each $j\in \{1,2\}$, let $S^j_i$ be the set of vertices $t\in V(T_j)$ such that $B^j_t$ contains a part coloured $i$. Observe that either $i$ is not used by $c_j$, in which case $S^j_i$ is empty, or $S^j_i$ induces a connected subgraph of $T_j$ (by \ref{strong2}) and $q_j\in S^j_i$. Let $S_i$ be the set of vertices $t\in V(T)$ such that $B_t$ contains a part coloured $i$. Observe that if $i$ is used by $c$, then $S_i=S_i^1\cup S_i^2\cup \{q\}$. By previous observations and since $q$ is adjacent to $q_1$ and $q_2$, it follows that $T[S_i]$ is connected. So \ref{strong2}, and the lemma, hold.

        Case \ref{OP2}:

        So there exist a cliquewidth-$k$ pair $(G,c')$ constructable with fewer operations and distinct $i,j\in \{1,\dots,k\}$, such that $i,j$ are used by $c'$ and $c$ is obtained from $c'$ by changing the colour of every vertex of colour $i$ to $j$. As $(G,c')$ can be built with fewer operations, induction applies. Let $\scr{P}$ and $(B_t:t\in V(T))$ be the partition of $G$ and tree-decomposition of $G/\scr{P}$ respectively obtained via induction. So $\scr{P}$ is dominated in $G$, is $c'$-monochromatic, and $(B_t:t\in V(T))$ has width at most $k-1$ and obeys \ref{strong1} and \ref{strong2} with respect to $c'$. 
        
        Observe that $\scr{P}$ is also $c$-monochromatic. Thus, we need only show that $(B_t:t\in V(T))$ obeys \ref{strong1} and \ref{strong2} with respect to $c$. 
        
        For \ref{strong1}, using induction, let $q\in V(T)$ be from \ref{strong1} with respect to $c'$. Let $C$ be the set of colours that are used by $c$, and let $C'$ be the set of colours that are used by $c'$. Since $j$ is used by $c'$, observe that $C=C'\setminus \{i\}$. By choice of $q$, for each $\ell\in C$, there exists a part $P_{\ell}$ in $B_q$ coloured $\ell$ under $c'$. Since $\ell\neq i$, $P_{\ell}$ is also coloured $\ell$ under $c$. It follows that $B_q$ is rainbow under $c_{\scr{P}}$. So \ref{strong1} is satisfied. 
        
        For \ref{strong2}, by induction and since $i$ is not used by $c$, observe that it suffices to consider only the colour $j$. For each $\ell\in \{i,j\}$, let $S_{\ell}$ be the set of vertices $t\in V(T)$ such that $B_t$ contains a part coloured $\ell$ under $c'$. By \ref{strong2}, $T[S_{\ell}]$ is a nonempty subtree of $T$. Let $S$ be the set of vertices $t\in V(T)$ such that $B_t$ contains a part coloured $j$ under $c$. Observe that $S=S_i\cup S_j$. Further, observe that $q\in S_i\cap S_j$ as $i,j$ are used by $c'$. Thus, since $T[S_i]$ and $T[S_j]$ induce nonempty subtrees of $T$ with a common vertex, $T[S]$ induces a nonempty subtree of $T$, as desired. So \ref{strong2}, and the lemma, hold.

        Case \ref{OP3}:

        So there exist a cliquewidth-$k$ pair $(G',c)$ constructable with fewer operations and distinct $i,j\in \{1,\dots,k\}$, such that $G'$ is a proper subgraph of $G$ and $G$ is obtained from $G'$ by adding an edge between each pair of (non-adjacent) vertices $u,v$ with $u$ coloured $i$ and $v$ coloured $j$. Note that because $G'$ is a proper subgraph of $G$, both $i$ and $j$ are used by $c$.
        
        As $(G',c)$ can be built with fewer operations, induction applies. Let $\scr{P}'$ and $(B'_t:t\in V(T))$ be the partition of $G'$ and tree-decomposition of $G'/\scr{P}'$ respectively obtained via induction. So $\scr{P}'$ is dominated in $G'$ (and $G$), is $c$-monochromatic, and $(B'_t:t\in V(T))$ has width at most $k-1$ and obeys \ref{strong1} and \ref{strong2}.

        For $\ell\in \{i,j\}$, let $\scr{P}_{\ell}\subseteq \scr{P}'$ be the set of parts coloured $\ell$ under $c$, and let $P_{\ell}:=\bigcup_{P\in \scr{P}_{\ell}}$. Observe that since $\ell$ is used by $c$, $P_{\ell}$ is nonempty. Set $\scr{P}:=\scr{P}\cup \{P_i,P_j\}\setminus (\scr{P}_i\cup \scr{P}_j)$. Observe that $\scr{P}$ is a partition of $G$ (as $P_i,P_j$ are nonempty).
        
        We first show that $\scr{P}$ is dominated in $G$ and $c$-monochromatic. By induction, observe that it suffices to show that $P_i,P_j$ are dominated in $G$ and $c$-monochromatic. The latter property is immediate; by definition, each vertex of $P_i$ is coloured $i$ under $c$, and each vertex of $P_j$ is coloured $j$ under $c$. So we focus on the former property. Let $\ell\in \{i,j\}$, and let $\ell'$ be the unique element of $\{i,j\}\setminus \{\ell\}$. Since $\ell'$ is used by $c$, there is at least one vertex $v_{\ell'}\in V(G)$ coloured $\ell'$ under $c$. Every vertex of $P_{\ell}$ is coloured $\ell$ under $c$, and is therefore adjacent to $v_{\ell'}$ in $G$ by construction. Thus, $P_{\ell}$ is dominated in $G$, as desired.

        We move to finding the desired tree-decomposition of $G/\scr{P}$. For each $t\in V(T)$, let $B_t$ be obtained from $B'_t$ by replacing each instance of a part $P\in \scr{P}_i$ with $P_i$, and each instance of a part $P\in \scr{P}_j$ with $P_j$, deleting duplicates as necessary. 
        
        Before we show that $(B_t:t\in V(T))$ is a tree-decomposition of $G/\scr{P}$, we first show that \ref{strong1} holds. Using induction, let $q\in V(T)$ be from \ref{strong1} with respect to $\scr{P}'$. By choice of $q$ and since $i,j$ are used by $c$, there exist a part of $\scr{P}'$ coloured $i$ and a part of $\scr{P}'$ coloured $j$ in $B'_q$. Hence, $P_i,P_j\in B_q$. From this and by induction, it follows that $B_q$ is rainbow under $c_{\scr{P}}$. So \ref{strong1} holds.
        
        We now show that $(B_t:t\in V(T))$ is a tree-decomposition of $G/\scr{P}$. Note that, by construction and subject to the previous statement being true, $(B_t:t\in V(T))$ has width at most the width of $(B_t:t\in V(T))$, which is at most at most $k-1$ by induction.
        
        For \ref{TD2}, by induction, observe that it suffices to consider edges in $G/\scr{P}$ with at least one endpoint in $\{P_i,P_j\}$. Further, observe that a part $P\in \scr{P} \setminus \{P_i,P_j\}$ is adjacent to $P_{\ell}$, $\ell\in \{i,j\}$, in $G$ if and only if there exists $P' \in \scr{P}'$ coloured $\ell$ adjacent to $P$ in $G'$. By induction, there exists $t\in V(T)$ such that $P,P'\in B'_t$. Thus, by construction, $P,P_{\ell}\in B_t$. Finally, observe that $P_i$ and $P_j$ are adjacent in $G$, so we must find $t\in V(T)$ such that $P_i,P_j\in B_t$. However, as observed above, $P_i,P_j\in B_q$, where $q$ is the vertex satisfying \ref{strong1} with respect to $\scr{P}'$. So we can pick $t:=q$. Thus, \ref{TD2} holds.

        For \ref{TD1}, by induction, observe that it suffices to consider only $P_i$ and $P_j$. For $\ell\in \{i,j\}$, let $S_{\ell}$ be the set of vertices $t\in V(T)$ such that $P_{\ell}\in B_t$. Observe that $S_{\ell}$ is precisely the set of vertices $t\in V(T)$ such that $B'_t$ contains a part coloured $\ell$. But by induction with \ref{strong2}, $T[S_{\ell}]$ induces a nonempty subtree of $T$. Hence, \ref{TD1} is satisfied, and $(B_t:t\in V(T))$ is a tree-decomposition of $G/\scr{P}$ of width at most $k-1$.

        Since we already showed that \ref{strong1} holds, it remains only to show that \ref{strong2} holds. By induction, observe that it suffices to only consider the colours $i$ and $j$. However, $P_i$ is the only part of $\scr{P}$ coloured $i$, and $P_j$ is the only part of $\scr{P}$ coloured $j$. Thus \ref{strong2} holds due to \ref{TD1}, and the lemma is true.

        We have shown that the lemma holds in all three cases, which completes the proof.
    \end{proof}

    \section{Quasi-isometry}
    \label{secQI}

    In this section, we complete the proof of \cref{QIcwtw}. We first need to introduce some definitions.

    The \defn{length} of a path is the number of edges in the path. The \defn{distance} between a pair of vertices $u,v$ in a graph $G$, denoted \defn{$\dist_G(u,v)$}, is the minimum length of a path between them (or $\infty$ if no such path exists). The \defn{distance} in $G$ between a pair $S,T\subseteq V(G)$, denoted \defn{$\dist_G(S,T)$}, is the minimum distance between a vertex in $S$ and a vertex in $T$.
    
    The \defn{weak diameter} of $S\subseteq V(G)$ in $G$ is the maximum distance in $G$ between a pair of vertices in $S$. The \defn{diameter} of $G$ is the weak diameter of $V(G)$ in $G$; the maximum distance between a pair of vertices in $G$.

    For a positive real number $c$ and graphs $G,H$, a map $f:V(G)\mapsto V(H)$ \defn{$c$-quasi-isometry} from $G$ to $H$ if the following hold:
    \begin{enumerate}[label=(QI\arabic*)]
        \item \label{QI1} for all $u,v\in V(G)$, $\dist_G(x,y)/c-c\leq \dist_H(f(x),f(y))\leq c\dist_G(x,y)+c$, and,
        \item \label{QI2} for all $x\in V(H)$, there exists $v\in V(G)$ such that $\dist_H(x,f(v))\leq c$.
    \end{enumerate}

    We say that $G$ is \defn{$c$-quasi-isometric} to $H$ if there exists a $c$-quasi-isometry from $G$ to $H$.

    Using \ref{QI2}, observe that we can define an inverse map $f^{-1}:V(H)\mapsto V(G)$ that is a $c'$-quasi-isometry from $H$ to $G$, where $c'$ is bounded by a function of $c$. So, up to a change in constants, quasi-isometry is a reflexive property.

    If \ref{QI1} holds but not \ref{QI2}, we instead say that $f$ is a \defn{$c$-quasi-isometric embedding} from $G$ to $H$. Say that $G$ is \defn{$c$-quasi-isometrically embeddable} into $H$ if there exists a $c$-quasi-isometric embedding from $G$ to $H$. Note that, unlike quasi-isometries, quasi-isometrically embeddability is not reflexive.

    Say that a class of graphs $\scr{G}$ is \defn{$c$-quasi-isometrically embeddable} in another class of graphs $\scr{H}$ if, for each $G\in \scr{G}$, there exists $H\in \scr{H}$ such that $G$ is $c$-quasi-isometrically embeddable in $H$. We remark that we do not define a similar parameter with quasi-isometries instead of quasi-isometric embeddings, as even if we insist that \ref{QI2} holds, we do not get a reflexive relationship between $\scr{G}$ and $\scr{H}$.

    We can now proceed with the proof of \cref{QIcwtw}. The key idea is that partitions with `small' (in terms of weak diameter) parts translate directly to quasi-isometries, as shown with the following lemma (which is nearly identical to an observation that appeared without proof in \citet{Hickingbotham2025}).

    \begin{lemma}
        \label{partQI}
        Let $c\geq 0$ be a real number, let $G$ be a graph, and let $\scr{P}$ be partition of $G$ such that each part has weak diameter at most $c$. Then $G$ is $(c+1)$-quasi-isometric to $G/\scr{P}$.
    \end{lemma}

    \begin{proof}
        Let $f:V(G)\mapsto V(G/\scr{P})$ be the map that sends each $v\in V(G)$ to the unique $P\in \scr{P}$ such that $v\in P$. We show that $f$ is a $(c+1)$-quasi-isometry. Since parts are nonempty, $f$ is injective. In particular, \ref{QI2} is trivially true. So we focus on \ref{QI1}.

        Fix $x,y\in V(G)$, and let $r:=\dist_G(x,y)$, $r':=\dist_{G/\scr{P}}(f(x),f(y))$. We must show that $r/(c+1)-(c+1)\leq r'\leq (c+1)r+(c+1)$. In fact, we show that $r/(c+1)-1\leq r'\leq r$. For the upper bound on $r'$, observe that, by triangle inequality, it suffices to consider the case when $x,y$ are adjacent. So $r=1$, and we must show $r'\leq 1$. However, this is immediate, as if two vertices in $G$ are adjacent, then the parts containing them are adjacent in $G/\scr{P}$. So we focus on the lower bound on $r'$.

        Let $f(x)=P_0,P_1,P_2,\dots,P_{r'}=f(y)$ be a path between $f(x)$ and $f(y)$ in $G/\scr{P}$ of length $r'$. By definition of the quotient, for each $i\in \{0,\dots,r'-1\}$, there exist vertices $y_i\in P_i$ and $x_{i+1}\in P_{i+1}$ that are adjacent in $G$. Set $x_0:=x$ and $y_{r'}:=y$. Since each part in $\scr{P}$ has weak diameter in $G$ at most $c$, for each $i\in \{0,\dots,r'\}$, we have that $\dist_G(x_i,y_i)\leq c$. By triangle inequality, it follows that $r\leq c(r'+1)+r'\leq (c+1)r'+(c+1)$. Thus, $r/(c+1)-1\leq r'$, as desired. So \ref{QI1}, and the lemma, hold.
    \end{proof}

    Observe that in a dominated partition, each part has weak diameter at most $2$. Hence, \cref{QIcwtw} follows as an immediate corollary of \cref{main,partQI}.

    \section{Proof of \cref{lowerBound}}
    \label{secLB}
    In this section, we prove \cref{lowerBound}. Neither the counterexample nor the proof is deep; we just show that a sufficiently subdivided complete graph has the desired properties. However, the proof is a little technical, and is aided by including some preparatory lemmas.

    We quickly give some basic definitions.

    A \defn{spider} is a tree with at most one vertex of degree at least 3.

    For a real number $r\geq 0$ and a graph $G$, the \defn{closed $r$-neighbourhood} of a set $S\subseteq V(G)$, denoted \defn{$N^r_G[S]$}, is the set of all vertices at distance at most $r$ from $S$ in $G$. Note that $S\subseteq N^r_G[S]$.
    
    For an integer $n\geq 0$, to \defn{subdivide} an edge $uv$ in a graph $G$ $n$ times is to replace $uv$ with a path from $u$ to $v$ of length exactly $n+1$, and to \defn{subdivide} $uv$ at least $n$ times is to replace $uv$ with a path from $u$ to $v$ of length at least $n+1$. The \defn{$n$-subdivision} of $G$ is the graph obtained from $G$ by subdividing each edge exactly $n$ times, and a \defn{$\geq n$-subdivision} is any graph obtained from $G$ by subdividing each edge at least $n+1$ times. A graph is a \defn{subdivision} of $G$ if it is a $\geq 0$-subdivision of $G$.

    A graph $H$ is a \defn{minor} of a graph $G$ if there exists connected and pairwise disjoint subgraphs $\{X_v:v\in V(H)\}$ such that $X_v$ is adjacent to $X_{v'}$ whenever $vv'\in E(H)$. Note that if we can find connected subgraphs $\{X_v:v\in V(H)\}\cup \{P_e:e\in E(H)\}$ such that (1) both $\{X_v:v\in V(H)\}$ and $\{P_e:e\in E(H)\}$ are pairwise disjoint collections, and (2) $P_e$, $e\in E(H)$, intersects $X_v$, $v\in V(H)$, if and only if $v$ is an endpoint of $e$, then $H$ is a minor of $G$.

    We can now proceed to the proof of \cref{lowerBound}.

    \begin{lemma}
        \label{cwPath}
        Let $n\geq 3$ be an integer, and let $i,j,k\in \{1,\dots,n\}$ be such that $j\notin \{i,k\}$ and $\{1\dots,n\}\setminus \{i,j,k\}\neq \emptyset$. Let $P$ be a path on at least two vertices, and let $c$ be a colouring of $P$ obtained from giving one endpoint $x$ the colour $i$, the other endpoint $y$ the colour $j$, and all other vertices the colour $k$. Then $(P,c)$ is a cliquewidth-$n$ pair.
    \end{lemma}
    \begin{proof}
        By induction on $|V(P)|$. The base case is when $|V(P)|=2$. Since $i\neq j$ and since $i,j\in \{1,\dots,n\}$, this case is trivial; start with the respective singleton graphs with the correct colours, and then use \ref{OP1} and \ref{OP3}. So we proceed to the inductive step with $|V(P)|\geq 3$.

        Since $\{1\dots,n\}\setminus \{i,j,k\}\neq \emptyset$, there exists some $\ell \in \{1\dots,n\}$ with $\ell \notin \{i,j,k\}$. Note that since $j\notin \{i,k\}$, $j\in \{1,\dots,n\}\setminus \{i,k,\ell\}$.
        
        Let $P':=P-y$. Since $|V(P)|\geq 3$, $|V(P')|\geq 2$. Thus, $P'$ has an endpoint $y'$ distinct from $x$. Let $c'$ be the colouring of $P'$ obtained from giving $x$ the colour $i$, $y'$ the colour $\ell$, and all other vertices the colour $k$. By induction $(P',c')$ is a cliquewidth-$n$ pair. 

        Observe that $(P,c)$ can be obtained from $(P',c')$ by adding a new vertex $y$ coloured $j$, adding the edge $yy'$, and then recolouring $y'$ to $k$. We can achieve the first step with \ref{OP1}, and then the second step with \ref{OP3} with the colours $j$ and $\ell$ as $j\neq \ell$, and since $y$ is the only vertex coloured $j$ and $y'$ is the only vertex coloured $\ell$. Finally, we can achieve the third step with \ref{OP2}, changing all the colour of all vertices coloured $\ell$ to $k$, as $y'$ is the only vertex coloured $\ell$. Thus, $(P,c)$ is a cliquewidth-$n$ pair, as desired.
    \end{proof}

    We remark that \cref{cwPath} gives the following (well-known) result as a corollary.
    \begin{corollary}
        \label{cwPathRed}
        Every path has cliquewidth at most $3$.
    \end{corollary}
    \begin{proof}
        This is trivial for paths of length 1. For length at least two, apply \cref{cwPath} with $n=3$, $i=k=1$, and $j=2$.
    \end{proof}

    \begin{lemma}
        \label{cwSpider}
        Let $t\geq 3$ be an integer, let $S$ be a spider with exactly $t$ leaves, and let $c$ be the colouring of $S$ defined by giving each leaf a distinct colour in $\{1,\dots,t\}$, and all other vertices the colour $t+1$. Then $(S,c)$ is a cliquewidth-$(t+3)$ pair.
    \end{lemma}

    \begin{proof}
        Let $r$ be the unique vertex of degree at least $3$ (which must exist as $S$ has at least three leaves). Assume that the leaves are labelled $1,\dots,t$, matching the colour they received. Let $L$ be the set of leaves adjacent to $r$.
        
        For each $i\in \{1,\dots,t\}$, let $P_i'\subseteq S$ be the path from $r$ to $i$, and let $P_i:=P_i'-r$. If $|V(P_i)|=1$ (which occurs if and only if $i\in L$), let $y_i:=i$, and let $c_i$ be the colouring of $P_i$ obtained from giving $i$ the colour $i$. In this case, $(P_i,c_i)$ is trivially a cliquewidth-$(t+3)$ pair. Otherwise, let $y_i$ be the endpoint of $P_i$ different from $i$, and let $c_i$ be the colouring of $P_i$ obtained from giving $i$ the colour $i$, $y_i$ the colour $t+2$, and all other vertices the colour $t+1$. Observe that $\{1,\dots,t+2\}\setminus \{i,t+1,t+2\}$ is nonempty as $t\geq 3$, and that $t+2\notin \{i,t+1\}$. Thus, by \cref{cwPath} with $n:=t+2$, $x:=i$, $y:=y_i$, $j:=t+2$, and $k:=t+1$, $(P_i,c_i)$ is a cliquewidth-$(t+2)$ pair (and thus also a cliquewidth-$(t+3)$ pair).

        Let $P_0$ be the graph consisting of only the vertex $r$, and let $c_0$ be the colouring of $P_0$ where $r$ is coloured $t+3$. $(P_0,c_0)$ is trivially a cliquewidth-$(t+3)$ pair. Let $S'$ be the disjoint union of $P_0,P_1,\dots,P_t$, and let $c':=c_0\cup c_1\cup \dots \cup c_t$. By repeated use of \ref{OP1}, $(S',c')$ is a cliquewidth-$(t+3)$ pair. Observe that $S$ can be obtained from $S'$ by adding, for each $i\in \{1,\dots,t\}$, an edge $ry_i$, and that $c$ can be obtained from $c'$ by recolouring $r$ to $t+1$, and, for each $i\in \{1,\dots,t\}\setminus L$, recolouring $y_i$ to $t+1$. We can achieve the former by use repeated use of \ref{OP3}. First, with the colours $t+3$ and $t+2$, as $r$ is the only vertex coloured $t+3$, and the set of vertices coloured $t+2$ is precisely $\{y_i:i\in \{1,\dots,t\}\setminus L\}$. Then, for each $i\in L$, with the colours $t+3$ and $i$, as $i=y_i$ is the only vertex coloured $i$. Finally, we can achieve the latter by two uses of \ref{OP2}, changing the colour all the vertices coloured $t+3$, and then $t+2$, to $t+1$. So $(S,c)$ is a cliquewidth-$(t+3)$ pair, as desired.
    \end{proof}

    We can now show that the first property in \cref{lowerBound}, the bound on the cliquewidth, is satisfied.

    \begin{lemma}
        \label{cwKn}
        Any subdivision $G$ of $K_n$, $n\geq 4$ an integer, has cliquewidth at most $n+2$.
    \end{lemma}

    \begin{proof}
        Let $X$ be the vertices of $G$ of degree at least $3$. We consider the vertices in $X$ to be labelled $1,\dots,t$. 
        
        Let $S$ be the spider obtained from taking the union of all the paths from $n$ to $\{1,\dots,n-1\}$ with no internal vertices in $\{1,\dots,n\}$. So $S$ is a spider with $n-1\geq 3$ leaves. Let $c$ be the colouring of $S$ obtained from giving each $i\in \{1,\dots,n\}$ the colour $i$ and all other vertices the colour $n$. By \cref{cwSpider} with $t:=n-1$, $S$ is a cliquewidth-$(n+2)$ pair.

        For each pair $i,j\in \{1,\dots,n-1\}$ with $i<j$, let $P_{i,j}$ be the path from $i$ to $j$ with no internal vertices in $\{1,\dots,n\}$, and let $P_{i,j}'$ be the (possibly empty) interior of $P_{i,j}$. If $|V(P'_{i,j})|=1$, let $v$ be the unique vertex of $P'_{i,j}$, set $x_{i,j}:=y_{i,j}:=v$, and let $c'_{i,j}$ be the colouring of $P_{i,j}'$ obtained by giving $v$ the colour $n+1$. It is immediate that $(P_{i,j}',c_{i,j}')$ is a cliquewidth-$(n+2)$ pair. If instead $|V(P'_{i,j})|\geq 2$, then let $c'_{i,j}$ be the colouring of $P_{i,j}'$ obtained from giving the vertex $x_{i,j}$ adjacent to $i$ the colour $n+1$, the vertex $y_{i,j}$ adjacent to $j$ the colour $n+2$, and every other vertex the colour $n$. Since $\{1,\dots,n+2\}\setminus \{n,n+1,n+2\}\neq \emptyset$ as $n\geq 4$, by \cref{cwPath} with $n:=n+2,x:=x_{i,j}$, $y:=y_{i,j}$, $i:=n+1$, $j:=n+2$, and $k:=n$, we obtain that $(P_{i,j}',c_{i,j}')$ is a cliquewidth-$(n+2)$ pair.
        
        For each pair $i,j\in \{0,\dots,n-1\}$ with $i\leq j$, let $G_{i,j}:=S\cup \bigcup_{k=1}^i\bigcup_{\ell=k+1}^j P_{k,\ell}$, and let $c_{i,j}$ be the colouring of $G_{i,j}$ obtained by giving each $i\in \{1,\dots,n\}$ the colour $i$, and all other vertices the colour $n$. Observe that $G_{n-1,n-1}=G$ and that, for each $j\in \{0,\dots,n-1\}$, $G_{0,j}=S$ and $c_{0,j}=c$. Also, note that for each $i\in \{1,\dots,n-1\}$, $G_{i-1,i}=G_{i,i}$ and $c_{i-1,i}=c_{i,i}$. Finally, observe that, for each pair $i,j\in \{1,\dots,n-1\}$ with $i>j$, if $|V(P'_{i,j})|\geq 1$, then $G_{i,j}$ can be obtained from the disjoint union of $G_{i,j-1}$ and $P'_{i,j}$ by adding the edges $ix_{i,j}$ and $jy_{i,j}$, and $c_{i,j}$ can be obtained from $c_{i,j-1}\cup c'_{i,j}$ by recolouring both $x_{i,j}$ and $y_{i,j}$ to $n$. If instead $|V(P'_{i,j})|=0$, then $G_{i,j}$ can be obtained from $G_{i,j-1}$ by adding the edge $ij$, and $c_{i,j-1}=c_{i,j}$.

        \begin{claim}
            \label{claimCwKn}
            For each pair $i,j\in \{0,\dots,n-1\}$ with $i\leq j$, $(G_{i,j},c_{i,j})$ is a cliquewidth-$(n+2)$ pair.
        \end{claim}
        \begin{proof}
            By induction on $i+j$. Observe that $i+j\geq 0$, and if $i+j=0$, then $i=j=0$. In this case, $G_{i,j}=S$ and $c_{i,j}=c$. Thus, $(G_{i,j},c_{i,j})$ is a cliquewidth-$(n+2)$ pair, as desired. Additionally, observe that the same argument applies anytime $i=0$. So we proceed to the inductive step, while assuming that $i\geq 1$.

            If $i=j$, then set $i':=i-1$. Observe that $i'+j< i+j$, so induction applies to $(G_{i',j},c_{i',j})$. However, since $i=j=i'+1$ and as observed above, $(G_{i',j},c_{i',j})=(G_{i,j},c_{i,j})$. So $(G_{i,j},c_{i,j})$ is a cliquewidth-$(n+2)$ pair, as desired.

            Otherwise, assume $i>j$. We need to consider three cases based on $|V(P_{i,j}')|$. 
            
            If $|V(P_{i,j}')|=0$, then, as observed above, $c_{i,j-1}=c_{i,j}$, and $G_{i,j}$ can be obtained from $G_{i,j-1}$ by adding the edge $ij$. Since $i$ is the only vertex coloured $i$ and since $j$ is the only vertex coloured $j$, this can be achieved using \ref{OP2} (as $i\neq j$). It follows that $(G_{i,j},c_{i,j})$ is a cliquewidth-$(n+2)$ pair, as desired. 
            
            If $|V(P_{i,j}')|=1$, then $G_{i,j}$ can be obtained from the disjoint union of $G_{i,j-1}$ and $P'_{i,j}$ by adding the edges $ix_{i,j}$ and $jy_{i,j}=jx_{i,j}$, and $c_{i,j}$ can be obtained from $c_{i,j-1}\cup c'_{i,j}$ by recolouring $x_{i,j}=y_{i,j}$ to $n$. Since $i$ is the only vertex coloured $i$, since $j$ is the only vertex coloured $j$, and since $x_{i,j}$ is the only vertex coloured $n+1$, the former can be achieved by \ref{OP1} and then two uses of \ref{OP3}; first with the colours $i$ and $n+1$, and then with the colours $j$ and $n+1$. We can then achieve the latter with \ref{OP2}, changing the colour of each vertex coloured $n+1$ to $n$. It follows that $(G_{i,j},c_{i,j})$ is a cliquewidth-$(n+2)$ pair, as desired.

            If $|V(P_{i,j}')|\geq 2$, then $G_{i,j}$ can be obtained from the disjoint union of $G_{i,j-1}$ and $P'_{i,j}$ by adding the edges $ix_{i,j}$ and $jy_{i,j}$, and $c_{i,j}$ can be obtained from $c_{i,j-1}\cup c'_{i,j}$ by recolouring both $x_{i,j}$ and $y_{i,j}$ to $n$. Since $i$ is the only vertex coloured $i$, since $j$ is the only vertex coloured $j$, since $x_{i,j}$ is the only vertex coloured $n+1$, and since $y_{i,j}$ is the only vertex coloured $n+2$, the former can be achieved by \ref{OP1} and then two uses of \ref{OP3}; first with the colours $i$ and $n+1$, and then with the colours $j$ and $n+2$. We can then achieve the latter with two uses of \ref{OP2}, first changing the colour of every vertex coloured $n+1$ to $n$, and then changing the colour of every vertex coloured $n+2$ to $n$. It follows that $(G_{i,j},c_{i,j})$ is a cliquewidth-$(n+2)$ pair, as desired.

            Thus, in all cases, $(G_{i,j},c_{i,j})$ is a cliquewidth-$(n+2)$ pair. This completes the inductive step, and the proof of the claim. 
        \end{proof}

        The lemma follows from \cref{claimCwKn} as $G_{n-1,n-1}=G$.
    \end{proof}

    We remark that any subdivision of $K_2$ (a path) has cliquewidth at most $3$ by \cref{cwPathRed}, and that an easy argument using \cref{cwPath} shows that any subdivision of $K_3$ (a cycle) has cliquewidth at most $4$.

    It remains only to show that second property in \cref{lowerBound}, being the treewidth of any graph it is quasi-isometric to. This follows from the relationship between subdivision, quasi-isometry, and minors. 
    
    The following lemma is likely already well-known, but we include a proof for completeness.

    \begin{lemma}
        \label{qiMinor}
        Let $c\geq 1$ be a real number, let $H$ be a graph, let $G$ be a $\geq (4c(c+1)-1)$-subdivision of $H$, and let $G'$ be a graph that $G$ $c$-quasi-isometrically embeds into. Then $G'$ contains $H$ as a minor.
    \end{lemma}

    \begin{proof}
        Let $f:V(G)\mapsto V(G')$ be a $c$-quasi-isometric embedding. Note that adjacent vertices in $H$ are at distance at least $4c(c+1)$ in $G$.
        
        For each $v\in V(H)$, let $X_v:=N_G^{c(c+1)}[v]$. For each pair of distinct $v,v'\in V(H)$, observe that $\dist_G(X_v,X_{v'})\geq 4c(c+1)-2c(c+1)=2c(c+1)$. 
        
        For each $vv'\in E(H)$, let $P_{vv'}$ be the vertices of the unique path in $G$ between $X_v$ and $X_{v'}$ with no internal vertices in $V(H)$. Note that $P_{vv'}$ intersects $X_v$ and $X_{v'}$, and that $P_{vv'}$ is at distance exactly $c(c+1)$ from both $v$ and $v'$. Thus, observe that for distinct $vv',ww'\in E(H)$, $\dist_G(P_{vv'},P_{ww'})\geq 2c(c+1)$. Also, note that for each $vv'\in E(H)$ and each $w\in V(H)$ that is not an endpoint of $ww'$, $\dist_G(P_{vv'},X_w)\geq 2c(c+1)$.

        For each $v\in V(H)$, let $X'_v:=N_{G'}^c[f(X_v)]$. Since $G[X_v]$ is connected and since adjacent vertices are mapped to vertices at distance at most $2c$ by \ref{QI1}, $G'[X'_v]$ is connected. Further, for each pair of distinct $v,v'\in V(H)$, by \ref{QI1}, observe that $\dist_{G'}(f(X_v),f(X_{v'}))\geq 2c(c+1)/c-c=2c+1$, and thus $\dist_{G'}(X_v',X_{v'}')\geq 2c+1-2c=1$. In particular, $X'_v$ and $X'_{v'}$ are disjoint.
        
        For each $vv'\in E(H)$, let $P'_{vv'}:=N_{G'}^c[f(P_{vv'})]$. Using the same argument as above, $G'[P'_{vv'}]$ is connected, and for each pair of distinct $vv',ww'\in E(H)$, $P'_{vv'}$ and $P'_{ww'}$ are disjoint. Similarly, for each $vv'\in E(H)$ and each $w\in V(H)$ that is not an endpoint of $vv'$, $P'_{vv'}$ and $X'_w$ are disjoint. Finally, note that, for each $vv'\in E(H)$, $P'_{vv'}$ intersects both $X'_v$ and $X'_{v'}$.

        Thus, we have sets $\{X'_v:v\in V(H)\}$, $\{P'_{vv'}:vv'\in E(H)\}$ that a) all induce connected subgraphs of $G'$ and b) intersect if and only if they form a incident edge/endpoint vertex pair. It follows from an observation made at the beginning of \cref{secLB} that $H$ is a minor of $G'$.
    \end{proof}

    \begin{corollary}
        For every real number $c\geq 1$ and each integer $n\geq 1$, each $\geq (4c(c+1)-1)$-subdivision of $K_n$ is not $c$-quasi-isometrically embeddable into any graph of treewidth less than $n-1$.
    \end{corollary}

    \begin{proof}
        Let $G$ be a $\geq (4c(c+1)-1)$-subdivision of $K_n$, and let $G'$ be a graph that $G$ $c$-quasi-isometrically embeds into. By \cref{qiMinor}, $G'$ has a $K_n$-minor. Since $K_n$ has treewidth $n-1$, and since treewidth does not increase under taking minor, $G'$ has treewidth at least $n-1$. The corollary follows.
    \end{proof}

    We can now complete the proof of \cref{lowerBound}.

    \begin{proof}
        Let $G$ be the $(4c(c+1)-1)$-subdivision of $K_{k-2}$. Since $k-2\geq 4$ as $k\geq 6$, by \cref{cwKn}, $G$ has cliquewidth at most $k$. By \cref{qiMinor} and since all $c$-quasi-isometries are also $c$-quasi-isometric embeddings, any graph that $G$ is $c$-quasi-isometric to has treewidth at least $k-3$.
    \end{proof}

    \section{Assouad--Nagata dimension}
    \label{secANdim}

    This section is devoted to proving \cref{ANdimCw}. We start with some definitions.

    For $r\in \ds{R}$ and a graph $G$, $S,T\subseteq V(G)$ are \defn{$r$-disjoint} if for all $u\in S$, $v\in T$, $\dist_G(u,v)>r$. A collection $\scr{C}$ of subsets of $V(G)$ is \defn{$r$-disjoint} if the sets in $\scr{C}$ are pairwise $r$-disjoint.

    A function $f:\ds{R}\rightarrow \ds{R}$ is a \defn{dilation} if there exists $c\in \ds{R}$ such that $f(r)=cr$ for all $r\in \ds{R}$. $f$ is \defn{increasing} if, for all $x,y\in \ds{R}$ with $x>y$, $f(x)>f(y)$.

    For an integer $n\geq 0$, a function $d:\ds{R}\rightarrow \ds{R}$ is an \defn{$n$-dimensional control function} for a class of graphs $\scr{G}$ if, for every $G\in \scr{G}$ and every $r\in \ds{R}$, there exist $n+1$ collections $\scr{C}_1,\dots,\scr{C}_{n+1}$ of subsets of $V(G)$ such that:

    \begin{enumerate}[label=(CF\arabic*)]
        \item \label{CF1}$\bigcup_{i=1}^{n+1}\bigcup_{S\in \scr{C}_i}S = V(G)$,
        \item \label{CF2}for each $i\in \{1,\dots,n+1\}$, $\scr{C}_i$ is $r$-disjoint, and
        \item \label{CF3} for each $i\in \{1,\dots,n+1\}$ and each $S\in \scr{C}_i$, $S$ has weak diameter in $G$ at most $d(r)$.
    \end{enumerate}

    The \defn{Assouad--Nagata} dimension of a class of graphs $\scr{G}$, \defn{$\ANdim(\scr{G})$}, is the minimum $n$ such that there exists a dilation that is an $n$-dimensional control function for $\scr{G}$ (or $\infty$ if no such $n$ exists).

    We can now start the proof of \cref{ANdimCw}. The key idea is that quasi-isometric embeddings can be used to bound the Assouad--Nagata dimension, as shown with the following lemma (which we suspect is well-known, but include a proof for completeness).

    \begin{lemma}
        \label{QIANdim}
        Let $c\in \ds{R}$, and let $\scr{G},\scr{H}$ be classes of graphs such that $\scr{G}$ is $c$-quasi-isometrically embeddable in $\scr{H}$. Then $\ANdim(\scr{G})\leq \ANdim(\scr{H})$.
    \end{lemma}

    \begin{proof}
        Let $n:=\ANdim(\scr{H})$. We may assume that $n<\infty$ as otherwise the lemma is trivially true. So there exists an $n$-dimensional control function $d'$ for $\scr{H}$ that is a dilation. Without loss of generality, we may assume that $d'$ is increasing, and that $d'(r)\geq 0$ for each $r\in \ds{R}$. Let $d:=r\mapsto cd'(2cr)+c^2r$. Observe that since $d'$ is a dilation, so is $d$. Thus, it suffices to show that $d$ is an $n$-dimensional control function for $\scr{G}$.

        Fix $r\in \ds{R}^+$ and $G\in \scr{G}$. Observe that if $r<1$, any pair of disjoint sets of vertices are $r$-disjoint. It follows that if $\scr{P}$ is the partition of $G$ into singletons, then $\{\scr{P}\}$ satisfies \ref{CF1}-\ref{CF3} (as $d'(r)\geq 0$). So we may assume $r\geq 1$. 
        
        Set $r':=cr+c$. Since $r\geq 1$ and since $d'$ is increasing, observe that $cd'(r')+c^2\leq d(r)$.
        
        Since $\scr{G}$ is $c$-quasi-isometrically embeddable in $\scr{H}$, there exist $H\in \scr{H}$ and a function $f:V(G)\mapsto V(H)$ that is a $c$-quasi-isometric embedding from $G$ into $H$.
        By definition of $d'$, there exist $n+1$ collections $\scr{C}'_1,\dots,\scr{C}'_{n+1}$ of subsets of $V(H)$ such that:
        \begin{enumerate}[label=(CF\arabic*-H)]
            \item \label{CFH1} $\bigcup_{i=1}^{n+1}\bigcup_{S\in \scr{C}'_i}S = V(H)$,
            \item \label{CFH2} for each $i\in \{1,\dots,n+1\}$, $\scr{C}'_i$ is $r'$-disjoint in $H$, and
            \item \label{CFH3} for each $i\in \{1,\dots,n+1\}$ and each $S\in \scr{C}'_i$, $S$ has weak diameter in $H$ at most $d'(r')$.
        \end{enumerate}

        For each $i\in \{1,\dots,n+1\}$, let $\scr{C}_i:=\{f^{-1}(S):S\in \scr{C}'_i\}$. Observe that \ref{CF1} is satisfied immediately by \ref{CFH1}. Next, using \ref{QI1} with $f$ and by \ref{CFH2}, $\scr{C}_i$ is $r'/c-1=r$-disjoint in $G$. So \ref{CF2} is satisfied. Similarly, by \ref{QI1} and \ref{CFH3}, $f^{-1}(S)$ has weak diameter in $G$ at most $cd'(r')+c^2\leq d(r)$. Thus, \ref{CF3} is satisfied, and the lemma holds.
    \end{proof}

    This can be combined with \cref{ANdimTw} to prove \cref{ANdimCw}.

    \begin{proof}[Proof of \cref{ANdimCw}]
        Note that \cref{QIcwtw} implies that the class of all graphs of cliquewidth at most $k$ is $3$-quasi-isometrically embeddable in the class of all graphs of treewidth at most $k-1$, as a $3$-quasi-isometry is also a $3$-quasi-isometric embedding. Thus, the upper bound follows from \cref{QIcwtw,QIANdim,ANdimTw}. The lower bound follows from the (easy to prove) fact that any class containing graphs of arbitrarily large diameter have Assouad--Nagata dimension at least 1, and the fact that any path has cliquewidth at most 3 by \cref{cwPathRed}. 
    \end{proof}

    We remark that each connected component of a graph of cliquewidth at most $2$ has diameter at most $3$ (as it can be checked that $P_4$ has cliquewidth exactly $3$). So the class of all graphs of cliquewidth at most $2$ has Assouad--Nagata dimension 0.

    \paragraph{Acknowledgements.} The author thanks David Wood for his supervision and suggestions to improve this paper. The author thanks Robert Hickingbotham, Itai Benjamini, and Robert Brignall for comments on the original arXiv version of this paper.

    \bibliography{Ref.bib}

\end{document}